\documentclass{amsart}
\usepackage{latexsym}
\usepackage{amssymb}
\usepackage{epsfig}
\usepackage{amsthm}
\usepackage{latexsym}
\usepackage{longtable}
\usepackage{epsfig}
\usepackage{amsmath}
\usepackage{hhline}
\usepackage{color}
\usepackage{tikz-cd}
\newtheorem{theorem}{Theorem}[section]

\newtheorem{lm}[theorem]{Lemma}
\newtheorem{tr}[theorem]{Theorem}

\newtheorem{pr}[theorem]{Proposition}

\newtheorem{ex}[theorem]{Example}

\begin{document}
\title{A note on Schur-Weyl dualities for $GL(m)$ and $GL(m|n)$}
\begin{abstract}
We use a unified elementary approach to prove the second part of classical, mixed, super, and mixed super Schur-Weyl dualities for general linear groups and supergroups over an infinite ground field of arbitrary characteristic. These dualities describe the endomorphism algebras of the tensor space and mixed tensor space, respectively, over the group algebra of the symmetric group and the Brauer wall algebra, respectively.
Our main new results are the second part of the mixed Schur-Weyl dualities and mixed super Schur-Weyl dualities over an infinite ground field of positive characteristic.
\end{abstract}
\author{Franti\v sek Marko}
\address{The Pennsylvania State University, 76 University Drive, Hazleton, 18202 PA, USA}
\email{fxm13@psu.edu}
\maketitle

\section*{Introduction}

The classical Schur-Weyl duality is one of the cornerstones of the representation theory of algebraic groups. It connects the representation theory of the general linear and symmetric groups. This duality has been generalized to mixed tensor space settings and established for various groups and supergroups. This paper studies Schur-Weyl dualities over an infinite ground field $K$ of arbitrary characteristic.
Its main contribution is a characteristic-free elementary approach that unifies the second parts of the classical, mixed, super, and mixed super Schur-Weyl dualities over infinite ground fields. The results for the mixed and mixed super Schur-Weyl dualities in positive characteristics are new.

Let $K$ be an infinite field, $G$ be the general linear group $GL(m)$ or the general linear supergroup $GL(m|n)$, and $Dist(G)$ the distribution algebra of $G$. Let $V$ be a natural $G$-supermodule, $W=V^*$ be its dual, and $\Sigma_{r}$ be the symmetric group on $r$ elements. 
There are commuting actions of the distribution algebra $Dist(G)$ and the group algebra $K\Sigma_r$ on  $V^{\otimes r}$. 
The image of $Dist(GL(m))$ in $End_K(V^{\otimes r})$ is called the Schur algebra and is denoted by $S(m,r)$. 
The image of $Dist(GL(m|n))$ in $End_K(V^{\otimes r})$ is called the Schur superalgebra and is denoted by $S(m|n,r)$. 
The first part of the Schur-Weyl duality says the image of $K\Sigma_r$ in $End_K(V^{\otimes r})$ is $End_G(V^{\otimes r})^{op}$. The second part states that $End_{K\Sigma_r}(V^{\otimes r})$ is isomorphic to $S(m,r)$ if $G=GL(m)$ and to $S(m|n,r)$ if $G=GL(m|n)$.

Schur first established the Schur-Weyl duality for $GL(m)$ over the ground field $\mathbb{C}$ of complex numbers in \cite{schur}. For the survey of Schur-Weyl duality over a field of positive characteristic, please consult \cite{d}.
The proof of the first part of Schur-Weyl duality that uses the first fundamental theorem of invariant theory is given in \cite[Thm 4.1]{cp}.
A simple proof of the first part of Schur-Weyl duality in the case $m\geq r$ is provided in \cite[p.210]{cl}. 

The super Schur-Weyl duality over fields of characteristic zero was established in \cite{br} and \cite{ser}. The second part of super Schur-Weyl duality over fields of positive characteristic was established by N. Muir as described in \cite[2.3 (1)]{donkin}.

Next, consider the mixed tensor product space $V^{\otimes r} \otimes W^{\otimes s}$. 
Let $\delta=m-n$ and $B_{r,s}(\delta)$ be the walled Brauer algebra. 

There are commuting actions of the distribution algebra $Dist(G)$ and the Brauer wall algebra $B_{r,s}^{m,n}(\delta)$ on  $V^{\otimes r}\otimes W^{\otimes s}$. 
The image of $Dist(GL(m))$ in $End_K(V^{\otimes r}\otimes W^{\otimes s})$ is called the rational Schur algebra
and is denoted by $S(m,r,s)$. 
The image of $Dist(GL(m|n))$ in $End_K(V^{\otimes r}\otimes W^{\otimes s})$ is called the rational Schur superalgebra
and is denoted by $S(m|n;r,s)$. 
The first part of the mixed Schur-Weyl duality states the image of $B_{r,s}^{m,n}(\delta)$ in $End_K(V^{\otimes r}\otimes W^{\otimes s})$ is $End_G(V^{\otimes r}\otimes W^{\otimes s})$. The second part states that  $End_{B_{r,s}^{m,n}(\delta)}(V^{\otimes r}\otimes W^{\otimes s})$ is isomorphic to $S(m,r,s)$ if $G=GL(m)$ and to $S(m|n,r,s)$ if $G=GL(m|n)$. 

The mixed Schur-Weyl duality for $GL(m)$ over fields of characteristic zero was first established in \cite{dds}, and an alternative proof was given in \cite{tange}. The first part was proved using the first fundamental theorem of invariant theory (see \cite[Thm 3.1]{cp}), 
while the second part was called ``the hard part'' in \cite{tange}. 
We are not aware of analogous results in the positive characteristic case.

The first part of the mixed super Schur-Weyl duality was proved over the ground field $K$ of characteristic zero in
\cite{bs} (see also \cite{cw}) following the approach specified in \cite{n}. In particular, by Theorem 7.8 of \cite{bs}, if $(m+1)(n+1)>r+s$, then $B_{r,s}^{m,n}(\delta)$ is isomorphic to $End_G(V^{\otimes r}\otimes W^{\otimes s})$.
Partial mixed super Schur-Weyl duality results were established earlier in \cite{sm}. We have learned recently that the second part of the mixed super Schur-Weyl duality was established in the case when 
$r+s\leq m-n$ in \cite{riesen}. 
We are unaware of a general proof of the second part of the mixed super Schur-Weyl duality in the characteristic zero case and of any evidence in the positive characteristic case.

This paper aims to give elementary proofs of the second part of the classical, super, mixed, and mixed super Schur -Weyl dualities over an infinite ground field of an arbitrary characteristic.
We obtain new results for the mixed Schur-Weyl dualities in the positive characteristic case and for the mixed super Schur-Weyl dualities in the case of arbitrary characteristic.

Since we use the algebraic independence of polynomial functions in several variables, we need to assume that the ground field is infinite. In general, this property is not true over finite fields. For Schur-Weyl duality over finite fields, see \cite{bd}.

In Section 1, we first introduce the Brauer wall algebra, actions of the general linear group $GL(m)$, and the general linear supergroup 
$GL(m|n)$ on ordinary tensor space $V^{\otimes r}$ and mixed tensor space $V^{\otimes r}\otimes W^{\otimes s}$.
We consider induced homomorphisms $\Psi_r$, $\Phi_r$, $\Psi_{r,s}$ and $\Phi_{r,s}$ and state classical, mixed, super and mixed super Schur-Weyl dualities.

In Sections 2 through 5, we use elementary methods to investigate the second part of Schur-Weyl dualities. In Section 2, we prove the second part of the classical Shur-Weyl duality, and in Section 3, we demonstrate the second part of the mixed Schur-Weyl duality, both involving the group $GL(m)$. In Section 4, we prove the second part of the super Schur-Weyl duality, and in Section 5, we derive the second part of the mixed super Schur-Weyl duality, both involving the supergroup $GL(m|n)$.
In Section 6, we show the limitation of this elementary approach to proving the first part of the classical Schur-Weyl duality.
\section{Notation}\label{section1}

Throughout the paper, we assume that the ground field $K$ is an infinite field of characteristic zero or positive characteristic $p$.

\subsection{Brauer wall algebra $B_{r,s}(\delta)$}
Let $\delta$ be an element of $K$.
We define generators of $B_{r,s}(\delta)$ by (the isotopy classes of) diagrams. A diagram is a bipartite graph on the set of upper vertices 
$\overline{1}, \ldots, \overline{r+s}$ and lower vertices $\underline{1}, \ldots, \underline{r+s}$ such that each vertex is connected to exactly one other vertex by an edge. The edges can be horizontal or vertical. 
Each vertical edge connects an upper vertex from the set $\overline{1}, \ldots, \overline{r}$ to a lower vertex 
from the set $\underline{1}, \ldots, \underline{r}$, or an upper vertex from the set $\overline{r+1}, \ldots, \overline{r+s}$ to a lower vertex from the set $\underline{r+1}, \ldots, \underline{r+s}$. 
Each horizontal edge connects an upper vertex from $\overline{1}, \ldots, \overline{r}$ to an upper vertex
from $\overline{r+1}, \ldots, \overline{r+s}$, or a lower vertex from $\underline{1}, \ldots, \underline{r}$ to a lower vertex 
from $\underline{r+1}, \ldots, \underline{r+s}$. 

To define the multiplication of diagrams $\sigma$ and $\tau$ put the diagram $\sigma$ under $\tau$ and create their concatenation $\sigma \circ \tau$ by removing any internal cycles (appearing in the middle of the concatenation).
If there are $t$ internal cycles, then the product $\sigma \tau$ equals $\delta^t \sigma \circ \tau$. We extend this 
to $(\delta^a \sigma)(\delta^b \tau)=\delta^{a+b+t} \sigma\circ \tau$.

To a transposition $\sigma_i=(i,i+1)$ for $1\leq i\neq r\leq r+s$, denote the diagram $\tau_i$ obtained by connecting 
each upper vertex $\overline{j}$ to a lower vertex $\underline{\sigma_i(j)}$.

For $\sigma_r=(r,r+1)$, denote the diagram $\tau_r$ obtained by connecting vertex $\overline{j}$ to $\underline{j}$
for each $j\neq r,r+1$, connecting $\overline{r}$ to $\overline{r+1}$, and $\underline{r+1}$ to $\underline{r}$.

Then $\tau_1, \ldots, \tau_{r+s-1}$ are generators of $B_{r,s}(\delta)$.
 
The relations between $\tau_i$ are described in (2.3)---(2.6) of \cite{bs}.

The group algebra $K\Sigma_{r+s}$ can be described by permutation diagrams such that there is a vector space isomorphism $\text{flip}_{r,s}: K\Sigma_{r+s} \to B_{r,s}(\delta)$ as in (2.1) of \cite{bs}.

The above description of $B_{r,s}(\delta)$ is valid for an arbitrary $\delta$. However, since we work with $G=GL(m|n)$ in this paper, we always assume $\delta=m-n$.

\subsection{Actions and homomorphisms}

We consider two cases: the classical case when $G=GL(m)$ and the supercase when $G=GL(m|n)$.

The parity $|j|$ is defined as $|j|=0$ for $j=1, \ldots, m$ and $|j|=1$ for $j=m+1, \ldots, m+n$.

If $G=GL(m)$, then a matrix unit $e_{ij}\in Dist_1(G)$ acts on
basis elements $\{v_1, \ldots, v_m\}$ of the natural module $V$ via
$e_{ij}v_k=\delta_{jk} v_i$ and on the basis elements $\{v^*_1, \ldots, v^*_m\}$ of $W=V^*$ as
$e_{ij}v^*_k=-\delta_{ik}v^*_j$.

Analogously, if $G=GL(m|n)$, then a matrix unit $e_{ij}\in Dist_1(G)$ acts on
basis elements $\{v_1, \ldots, v_{m+n}\}$ of the natural module $V$ via
$e_{ij}v_k=\delta_{jk} v_i$ and on the basis elements $\{v^*_1, \ldots, v^*_{m+n}\}$ of $W=V^*$ as
$e_{ij}v^*_k=-\delta_{ik}(-1)^{|i|(|i|+|j|)}v^*_j$.

\subsubsection{Actions on $V^{\otimes r}$}
We write generators of $V^{\otimes r}$ as  
\[v_{i_1}\otimes \ldots \otimes v_{i_r} =v_I.\]

If $G=GL(m)$, then the right action of generators $\tau_j=(j,j+1)$, for $1\leq j<r$, of the group algebra $K\Sigma_r$ on generators of the tensor product space $V^{\otimes r}$ is given as 
$(v_I)\tau_j=v_J$, where the multi-index $J=I.(j,j+1)$ is obtained from $I$ by transposing entries at the $j$th and $(j+1)$st place.

If $G=GL(m|n)$, the right action of $K\Sigma_r$ on $V^{\otimes r}$ is given as \[(v_I)\tau_j=(-1)^{|i_j||i_{j+1}|}v_J=(-1)^{|i_j||i_{j+1}|}v_{I.(j,j+1)}.\]

The left action of $GL(m)$ on $V^{\otimes r}$ is given by 
\[\begin{aligned}g.v_I&=g.(v_{i_1}\otimes\ldots \otimes v_{i_r})=g.v_{i_1}\otimes \ldots  \otimes g.v_{i_r},
\end{aligned}
\]
and the left action of the general linear Lie algebra $\mathfrak{gl}(m)$ on $V^{\otimes r}$ is given as
\[\begin{aligned}e_{ij}.v_I&=e_{ij}.(v_{i_1}\otimes\ldots \otimes v_{i_r})=\sum_{a=1}^r v_{i_1}\otimes \ldots \otimes e_{ij}v_{i_a} \otimes \ldots \otimes v_{i_r}\\
&=\sum_{a=1}^r \delta_{ji_a} v_{i_1}\otimes \ldots \otimes v_{i} \otimes \ldots \otimes v_{i_r}.
\end{aligned}
\]

The left action of $GL(m|n)$ on $V^{\otimes r}$ is given analogously as 
\[\begin{aligned}&g.v_I=g.(v_{i_1}\otimes\ldots \otimes v_{i_r})=g.v_{i_1}\otimes \ldots  \otimes g.v_{i_r},
\end{aligned}
\]
and the left action of the general linear Lie superalgebra $\mathfrak{gl}(m|n)$ on $V^{\otimes r}$ is given as
\[\begin{aligned}e_{ij}.v_I&=e_{ij}.(v_{i_1}\otimes\ldots \otimes v_{i_r}) =
\sum_{a=1}^r (-1)^{|e_{ij}|(|i_1|+\ldots +|i_{a-1}|)}v_{i_1}\otimes \ldots \otimes e_{ij}v_{i_a} \otimes \ldots \otimes v_{i_r}\\
&=\sum_{a=1}^r \delta_{ji_a}(-1)^{|e_{ij}|(|i_1|+\ldots +|i_{a-1}|)}v_{i_1}\otimes \ldots \otimes v_{i} \otimes \ldots \otimes v_{i_r}.
\end{aligned}
\]

The right action of $K\Sigma_r$ on $V^{\otimes r}$ commutes with the left action of $G$ on $V^{\otimes r}$.

Therefore, there are induced homomorphisms
\begin{equation}\label{3} 
\Psi_{r}: Dist(G) \to End_{K\Sigma_r} (V^{\otimes r})^{op}
\end{equation}
and 
\begin{equation}\label{1}
\Phi_{r}: K\Sigma_{r} \to End_{G} (V^{\otimes r})^{op}\end{equation}

If $G=GL(m)$, the image of $Dist(G)$ in $End_{K\Sigma_r}(V^{\otimes r})^{op}$ is the Schur algebra $S(m,r)$.
If $G=GL(m|n)$, the image of  $Dist(G)$ in $End_{K\Sigma_r}(V^{\otimes r})^{op}$ is the Schur superalgebra $S(m|n,r)$.

\subsubsection{Actions on $V^{\otimes r}\otimes W^{\otimes s}$}
We write the generators of $V^{\otimes r}\otimes W^{\otimes s}$ as 
\[v_{i_1}\otimes \ldots \otimes v_{i_r} \otimes v^*_{i_{r+1}} \otimes \ldots \otimes v^*_{i_{r+s}}=v_I\,\]
where the multi-index $I=(i_1, \ldots, i_{r+s})$ has entries in the set $\{1, \ldots, m+n\}$.

The left actions of $GL(m)$ and $GL(m|n)$ on $V^{\otimes r}\otimes W^{\otimes s}$ are given  as

\[\begin{aligned}g.v_I&=g.(v_{i_1}\otimes\ldots \otimes v_{i_r}\otimes v^*_{i_{r+1}}\otimes \ldots \otimes v^*_{r+s})\\
&=g.v_{i_1}\otimes \ldots \otimes g. v_{i_r}\otimes g.v^*_{i_{r+1}}\otimes \ldots g.v^*_{i_{r+s}}.
\end{aligned}
\]

Also, the actions of $\mathfrak{gl}(m)$ and $\mathfrak{gl}(m|n)$ on $V^{\otimes r}\otimes W^{\otimes s}$ are given as 
\[\begin{aligned}e_{ij}.v_I&=e_{ij}.(v_{i_1}\otimes\ldots \otimes v_{i_r} \otimes v^*_{i_{r+1}}\otimes \ldots \otimes v^*_{i_{r+s}})\\
&=\sum_{a=1}^r v_{i_1}\otimes \ldots \otimes e_{ij}v_{i_a} \otimes \ldots \otimes v_{i_r} \otimes v^*_{i_{r+1}}\otimes \ldots \otimes v^*_{i_{r+s}}\\
&+\sum_{b=1}^s v_{i_1}\otimes \ldots \otimes v_{i_r} \otimes v^*_{i_{r+1}}\otimes \ldots \otimes e_{ij}v^*_{i_{r+b}}\otimes \ldots \otimes v^*_{i_{r+s}}\\
&=\sum_{a=1}^r \delta_{ji_a} v_{i_1}\otimes \ldots \otimes v_{i} \otimes \ldots \otimes v_{i_r} \otimes v^*_{i_{r+1}}\otimes \ldots \otimes v^*_{i_{r+s}}\\
&-\sum_{b=1}^s \delta_{ii_{r+b}} v_{i_1}\otimes \ldots \otimes v_{i_r} \otimes v^*_{i_{r+1}}\otimes \ldots \otimes v^*_{j}\otimes \ldots \otimes v^*_{i_{r+s}}
\end{aligned}
\]
and 
\[\begin{aligned}&e_{ij}.v_I=e_{ij}.(v_{i_1}\otimes\ldots \otimes v_{i_r} \otimes v^*_{i_{r+1}}\otimes \ldots \otimes v^*_{i_{r+s}})=\\
&\sum_{a=1}^r (-1)^{|e_{ij}|(|i_1|+\ldots +|i_{a-1}|)}v_{i_1}\otimes \ldots \otimes e_{ij}v_{i_a} \otimes \ldots \otimes v_{i_r} \otimes v^*_{i_{r+1}}\otimes \ldots \otimes v^*_{i_{r+s}}\\
&+\sum_{b=1}^s (-1)^{|e_{ij}|(|i_1|+\ldots +|i_{r+b-1}|)}v_{i_1}\otimes \ldots \otimes v_{i_r} \otimes v^*_{i_{r+1}}\otimes \ldots \otimes e_{ij}v^*_{i_{r+b}}\otimes \ldots \otimes v^*_{i_{r+s}}\\
&=\sum_{a=1}^r \delta_{ji_a} (-1)^{|e_{ij}|(|i_1|+\ldots +|i_{a-1}|)}v_{i_1}\otimes \ldots \otimes v_{i} \otimes \ldots \otimes v_{i_r} \otimes v^*_{i_{r+1}}\otimes \ldots \otimes v^*_{i_{r+s}}\\
&-\sum_{b=1}^s \delta_{ii_{r+b}}(-1)^{|i|(|i|+|j|)}(-1)^{|e_{ij}|(|i_1|+\ldots +|i_{r+b-1}|)}\\
&v_{i_1}\otimes \ldots \otimes v_{i_r} \otimes v^*_{i_{r+1}}\otimes \ldots \otimes v^*_{j}\otimes \ldots \otimes v^*_{i_{r+s}},
\end{aligned}
\]
respectively.

The right action of generators of the Brauer algebra $B_{r,s}(\delta)$ on generators of the mixed tensor product space $V^{\otimes r} \otimes W^{\otimes s}$ is given as follows.

Assume $G=GL(m)$. If $1\leq j\leq r+s-1$ and $j \neq r$, then we define
\[(v_I)\tau_j=v_J=v_{I.(j,j+1)},\] where the multi-index $J$ is obtained from $I$ by transposing entries at the $j$th and $(j+1)$st place. For $j=r$, we define
\[(v_I)\tau_r=\delta_{i_r,i_{r+1}} \sum_{k=1}^{r+s} v_{i_1}\otimes\ldots \otimes v_{i_{r-1}}\otimes 
v_k\otimes v^*_{k} \otimes v^*_{i_{r+2}}\otimes \ldots \otimes v^*_{i_{r+s}}.
\]

Assume $G=GL(m|n)$. If $1\leq j\leq r+s-1$ and $j \neq r$, then we define
\[(v_I)\tau_j= (-1)^{|i_j||i_{j+1}|}v_J= (-1)^{|i_j||i_{j+1}|}v_{I.(j,j+1)},\] where the multi-index $J$ is obtained from $I$ by transposing entries at the $j$th and $(j+1)$st place.

For $j=r$, we define
\[(v_I)\tau_r=-v^*_{i_{r+1}}(v_{i_r})(-1)^{|i_r|}\sum_{k=1}^{r+s} v_{i_1}\otimes\ldots \otimes v_{i_{r-1}}\otimes 
v_k\otimes v^*_{k} \otimes v^*_{i_{r+2}}\otimes \ldots \otimes v^*_{i_{r+s}}.
\]

\begin{lm}
The element $\tau=\sum_{k=1}^{r+s} v_k\otimes v^*_k$ spans a one-dimensional $G$-supermodule.
\end{lm}
\begin{proof}
Compute
\[\begin{aligned}&e_{ij}\tau=\sum_{k=1}^{r+s} [e_{ij} v_k\otimes v^*_k+(-1)^{|e_{ij}||v_k|}v_k\otimes e_{ij}v^*_k]=
e_{ij}v_j\otimes v^*_j+(-1)^{|e_{ij}||v_i|}v_i\otimes e_{ij}v^*_i\\
&=v_i\otimes v^*_j+(-1)^{|e_{ij}||v_i|}v_i\otimes
-(-1)^{|e_{ij}||v^*_i|}v^*_j=0
\end{aligned}\]
because $|v_i|=|v^*_i|$.
\end{proof}

\begin{lm}\label{above}
The right action of $B_{r,s}(\delta)$ on $V^{\otimes r}\otimes W^{\otimes s}$ commutes with the left action of $G$. 
\end{lm}
\begin{proof}
It is clear that the right action of $\Sigma_r\otimes \Sigma_s$ commutes with the left action of $G$.
Therefore, it is enough to verify that $g(v_I\tau_r)=(gv_I)\tau_r$ for each $g\in G$.
We have 
\[\begin{aligned}&(gv_I)\tau_r\\
&=\sum_{t=1}^{r} (-1)^{|g|(|i_1|+\ldots + |i_{t-1}|)}(v_{i_1}\otimes\ldots \otimes gv_{i_t}\otimes \ldots\otimes 
v_{i_r}\otimes v^*_{i_{r+1}} \otimes v^*_{i_{r+2}}\otimes \ldots \otimes v^*_{i_{r+s}})\tau_r\\
&+\sum_{t=r+1}^{r+s} (-1)^{|g|(|i_1|+\ldots + |i_{t-1}|)}(v_{i_1}\otimes \ldots\otimes v_{i_r}\otimes v^*_{i_{r+1}} \otimes \ldots \otimes gv^*_{i_t}\otimes \ldots \otimes v^*_{i_{r+s}})\tau_r\\
&=\sum_{t=1}^{r-1}(-1)^{|g|(|i_1|+\ldots + |i_{t-1}|)}(-1)^{|i_r|+1}v^*_{i_{r+1}}(v_{i_r})\\
&\sum_{k=1}^{r+s} v_{i_1}\otimes \ldots \otimes gv_{i_t}\otimes \ldots \otimes v_{i_{r-1}}\otimes v_k\otimes v^*_k\otimes v^*_{i_{r+2}}\otimes \ldots \otimes v^*_{i_{r+s}}\\
&+(-1)^{|g|(|i_1|+\ldots +|i_{r-1}|)}(-1)^{|g|+|i_r|+1}v^*_{i_{r+1}}(gv_{i_r})\\
&\sum_{k=1}^{r+s}v_{i_1}\otimes \ldots \otimes v_{i_{r-1}}\otimes v_k\otimes v^*_k\otimes v^*_{i_{r+2}}\otimes \ldots \otimes v^*_{i_{r+s}}\\
&+(-1)^{|g|(|i_1|+\ldots +|i_{r-1}|+|i_r|)}(-1)^{|i_r|+1}gv^*_{i_{r+1}}(v_{i_r})\\
&\sum_{k=1}^{r+s}v_{i_1}\otimes \ldots \otimes v_{i_{r-1}}\otimes v_k\otimes v^*_k\otimes v^*_{i_{r+2}}\otimes \ldots \otimes v^*_{i_{r+s}}\\
&+\sum_{t=r+2}^{r+s}(-1)^{|g|(|i_1|+\ldots + |i_{t-1}|)}(-1)^{|i_r|+1}v^*_{i_{r+1}}(v_{i_r})\\
&\sum_{k=1}^{r+s}v_{i_1}\otimes \ldots \otimes v_{i_{r-1}}\otimes v_k\otimes v^*_k\otimes v^*_{i_{r+2}}\otimes \ldots \otimes gv^*_{i_t}\otimes \ldots \otimes v^*_{i_{r+s}}.
\end{aligned}
\]

Since $gv^*_{i_{r+1}}(v_{i_r})=(-1)^{|g||i_{r+1}|+1}v^*_{i_{r+1}}(g^{-1}v_{i_r})$, the
expression
\[\begin{aligned}&(-1)^{|g|(|i_1|+\ldots +|i_{r-1}|+|i_r|)}(-1)^{|i_r|+1}gv^*_{i_{r+1}}(v_{i_r})\\
&=(-1)^{|g|(|i_1|+\ldots +|i_{r-1}|+|i_r|+|i_{r+1}|)}(-1)^{|i_r|}v^*_{i_{r+1}}(g^{-1}v_{i_r})
\end{aligned}
\]
is the opposite of 
\[(-1)^{|g|(|i_1|+\ldots +|i_{r-1}|)}(-1)^{|g|+|i_r|+1}v^*_{i_{r+1}}(g^{-1}v_{i_r})\] 
because if $v^*_{i_{r+1}}(g^{-1}v_{i_r})\neq 0$, then $|g|=|i_r|+|i_{r+1}|$.

Therefore, the two middle terms in the above sum offset each other, and 
\[\begin{aligned}&(gv_I)\tau_r\\
&=\sum_{t=1}^{r-1}(-1)^{|g|(|i_1|+\ldots + |i_{t-1}|)}(-1)^{|i_r|+1}v^*_{i_{r+1}}(v_{i_r})\\
&\sum_{k=1}^{r+s} v_{i_1}\otimes \ldots \otimes gv_{i_t}\otimes \ldots \otimes v_{i_{r-1}}\otimes v_k\otimes v^*_k\otimes v^*_{i_{r+2}}\otimes \ldots \otimes v^*_{i_{r+s}}\\
&+\sum_{t=r+2}^{r+s}(-1)^{|g|(|i_1|+\ldots + |i_{t-1}|)}(-1)^{|i_r|+1}v^*_{i_{r+1}}(v_{i_r})\\
&\sum_{k=1}^{r+s}v_{i_1}\otimes \ldots \otimes v_{i_{r-1}}\otimes v_k\otimes v^*_k\otimes v^*_{i_{r+2}}\otimes \ldots \otimes gv^*_{i_t}\otimes \ldots \otimes v^*_{i_{r+s}}.
\end{aligned}
\]
On the other hand,
\[\begin{aligned}&g(v_I\tau_r)\\
&=g(v^*_{i_{r+1}}(v_{i_r})(-1)^{|i_r|+1}\sum_{k=1}^{r+s}v_{i_1}\otimes \ldots \otimes v_{i_{r-1}}\otimes v_k\otimes v^*_k\otimes v^*_{i_{r+2}}\otimes \ldots \otimes v^*_{i_{r+s}})\\
&=\sum_{t=1}^{r-1}(-1)^{|g|(|i_1|+\ldots + |i_{t-1}|)}(-1)^{|i_r|+1}v^*_{i_{r+1}}(v_{i_r})\\
&\sum_{k=1}^{r+s} v_{i_1}\otimes \ldots \otimes gv_{i_t}\otimes \ldots \otimes v_{i_{r-1}}\otimes v_k\otimes v^*_k\otimes v^*_{i_{r+2}}\otimes \ldots \otimes v^*_{i_{r+s}}\\
&+\sum_{t=r+2}^{r+s}(-1)^{|g|(|i_1|+\ldots +|i_{r-1}|+|i_{r+2}|+\ldots + |i_{t-1}|)}(-1)^{|i_r|+1}v^*_{i_{r+1}}(v_{i_r})\\&\sum_{k=1}^{r+s}
v_{i_1}\otimes \ldots \otimes v_{i_{r-1}}\otimes v_k\otimes v^*_k\otimes v^*_{i_{r+2}}\otimes \ldots \otimes gv^*_{i_t}\otimes \ldots \otimes v^*_{i_{r+s}}\\
\end{aligned}\]
due to Lemma \ref{above}.
If $v^*_{i_{r+1}}(v_{i_r})\neq 0$, then $|i_{t+1}|=|i_t|$ and 
\[(-1)^{|g|(|i_1|+\ldots +|i_{r-1}|+|i_{r+2}|+\ldots + |i_{t-1}|)}=(-1)^{|g|(|i_1|+\ldots + |i_{t-1}|)}.\]

Therefore $g(v_I\tau_r)=(gv_I)\tau_r$.
\end{proof}

Therefore, there are induced homomorphisms
\begin{equation}\label{4} 
\Psi_{r,s}: Dist(G) \to End_{B_{r,s}(\delta)} (V^{\otimes r} \otimes W^{\otimes s})^{op}
\end{equation}
and 
\begin{equation}\label{5}
\Phi_{r,s}:B_{r,s}(\delta)\to End_G(V^{\otimes r}\otimes W^{\otimes s})^{op}.
\end{equation}

The image of $Dist(GL(m))$ in $End_{B_{r,s}(\delta)}(V^{\otimes r} \otimes W^{\otimes s})^{op}$ is, by definition, the rational Schur algebra $S(m,r,s)$. We denote $A(m,r,s)=S(m,r,s)^*$. 

The image of $Dist(GL(m|n))$ in $End_{B_{r,s}(\delta)}(V^{\otimes r} \otimes W^{\otimes s})^{op}$ is, by definition, the rational Schur superalgebra $S(m|n,r,s)$. We denote $A(m|n,r,s)=S(m|n,r,s)^*$. 

\subsection{Schur-Weyl dualities}
The Schur-Weyl dualities are often expressed as a double centralizer property. Their convenient formulation is stated as follows. The first part of the Schur-Weyl dualities states that the morphism $\Phi_r$, or $\Phi_{r,s}$ respectively, is surjective, while the second part of the Schur-Weyl dualities states that the morphism $\Psi_r$, or $\Psi_{r,s}$ respectively, is surjective.

The classical Schur-Weyl dualities over $\mathbb{C}$ were established by Schur in \cite{schur}. For a survey of results over fields of positive characteristic, consult \cite{d}.
If $r\leq m$, then $\Phi_r$ is injective and the first part of the classical Schur-Weyl duality can be reformulated as the isomophism $K\Sigma_r \simeq End_{GL(m)}(V^{\otimes r})^{op}$.
The second part of the classical Schur-Weyl duality states that \[S(m,r)\simeq End_{K\Sigma_r}(V^{\otimes r})^{op}.\]

The super Schur-Weyl duality over fields of characteristic zero were derived in \cite{br}.
The second part of the super Schur-Weyl duality states that \[S(m|n,r)\simeq End_{K\Sigma_r}(V^{\otimes r})^{op}.\]
The second part of the super Schur-Weyl duality over fields of positive characteristic was determined by N. Muir - see \cite{d}.

The mixed Schur-Weyl duality over fields of characteristic zero was proved in \cite{dds}.
If the characteristic of the ground field $K$ is zero and $r+s<(m+1)(n+1)$, then $\Phi_{r,s}$ is injective by \cite{bs}, and the first part of the mixed Schur-Weyl duality can be reformulated as \[KB_{r,s} \simeq End_{GL(m)}(V^{\otimes r}\otimes W^{\otimes s})^{op}.\]
The second part of the mixed Schur-Weyl duality states that \[S(m,r,s)\simeq End_{B_{r,s}}(V^{\otimes r}\otimes W^{\otimes s})^{op}.\]
We are unaware of general results on mixed Schur-Weyl duality over fields of positive characteristic.
We will show that the second part of mixed Schur-Weyl duality is valid over any infinite field of arbitrary characteristic.

While this paper was under review, we have learned that A. Riesen \cite{riesen} has established the second part of the mixed super Schur-Weyl duality over the fields of characteristic zero, stating that 
 \[S(m|n,r,s)\simeq End_{B_{r,s}(\delta)}(V^{\otimes r}\otimes W^{\otimes s})^{op}\]
under the condition $r+s\leq m-n$.
We will show that the second part of mixed super Schur-Weyl duality is valid over any infinite field of arbitrary characteristic.

To derive the second parts of Schur-Weyl dualities, we determine the structure of the duals of 
$End_{\Sigma_r}(V^{\otimes r})^{op}$ and $End_{B_{r,s}(\delta)}(V^{\otimes r} \otimes W^{\otimes s})^{op}$,
and compare it with $A(m,r)$, $A(m|n,r)$, $A(m,r,s)$ and $A(m|n,r,s)$, respectively.

\section{The second half of the Schur-Weyl duality for $GL(m)$}

The second part of Schur-Weyl duality states that $S(m,r)$ is isomorphic to $End_{K\Sigma_r}(V^{\otimes r})^{op}$. 

Let $I,J$ be multi-indices of length $r$ with entries from the set $\{1, \ldots, m\}$.
Denote by $E_{IJ}$ the matrix unit given by $E_{IJ}(v_K)=\delta_{JK} v_I$. The elements $E_{IJ}$ form a $K$-basis
of the space $End_K(V^{\otimes r})$. 
For $\phi\in End_K(V^{\otimes r})$, we write 
$\phi=\sum_{IJ} a_{IJ} E_{IJ}$ for appropriate coefficients $a_{IJ}$.

\begin{lm}\label{reduction}
If  $p>r$, then every morphism in $End_{K\Sigma_r}(V^{\otimes r})$ is a linear combination of morphisms $\phi=\psi\otimes \ldots \otimes \psi$, where $\psi\in End_{K}(V)$. If $\psi=\sum_{i,j}a_{ij}e_{ij}$, then 
$\phi=\sum_{IJ} a_{IJ} E_{IJ}$, where $a_{IJ}=a_{i_1j_1}\ldots a_{i_rj_r}$ for $I=i_1\ldots i_r$ and $J=j_1\ldots j_r$.
Additionally, a subspace of $V^{\otimes r}$ is $K\Sigma_r$-submodule if and only if it is invariant under the action of $G=GL(V)$.
\end{lm}
\begin{proof}
We adapt the classical arguments, e.g., \cite[Lemma 6.23 ]{fh}. 

Write $U=End_K(V)=V^*\otimes V$. Then $End_K(V^{\otimes r})=(V^*)^{\otimes r}\otimes V^{\otimes r}=(V^*\otimes V)^{\otimes r}$ with the compatible action of $\Sigma_r$.
Therefore $End_{K\Sigma_r}(V^{\otimes r})\simeq Sym^r (End_K(V))$.

Since $p>r$, the space $Sym^r (U)$ is spanned by $u^r=r! u\otimes \ldots \otimes u$ for $u\in U$. 
This follows from the polarization identity
\[\sum_{\sigma\in \Sigma_r} f_{\sigma(1)}\otimes f_{\sigma(2)}\otimes \ldots \otimes f_{\sigma(r)}=\sum_{I\subset \{1, \ldots, r\}} (-1)^{r-|I|}  (\sum_{i\in I} f_i)^{\otimes r}.\]

Therefore when we set $\psi=\frac{1}{\sqrt[n]{r!}}u$, we obtain $a_{IJ}=a_{i_1j_1}\ldots a_{i_rj_r}$.

The second statement follows because $GL(V)$ is dense in $End(V)$ either in Euclidean or Zariski topology.
\end{proof}

Note that even though the coefficients of the map $\phi=\psi\otimes\ldots\psi$ are built multiplicatively from the coefficients of $\psi$, it is difficult to work with the linear combinations of such maps $\phi$. 

We will follow the approach of \cite{dds} based on the following lemma. The multiplicative property we will use will follow from the particular choice of the basis $\mathcal{B}$.

\begin{lm} \label{switch} (Lemma 2.3 of \cite{dds})
Let $R$ be a commutative ring with 1. 
Let $M$ be a free $R$-module with basis $\mathcal{B}=\{b_1, \ldots,  b_l\}$ and $U$ a
submodule of $M$ given by a set of linear equations on the coefficients with respect
to the basis $\mathcal{B}$, i.e., $a_{ij}\in R$ such that 
$U=\{\sum c_ib_i \in M: \sum_j a_{ij}c_j=0$ for all i\} exist. 
Let $\{b_1^*, \ldots, b_l^*\}$ be the basis of $M^*=Hom_R(M, R)$ dual to $\mathcal{B}$, and let $X$ be
the submodule of $M^*$ generated by all $\sum_j a_{ij}b_j^*$.  
Then $U\simeq (M^*/X)^*$.
\end{lm}

\begin{pr}\label{classic} There is
$S(m,r)\simeq End_{K\Sigma_r}(V^{\otimes r})^{op}$.
\end{pr}
\begin{proof}
We show that the dual of $End_{K\Sigma_r}(V^{\otimes r})^{op}$ is isomorphic to the coalgebra $A(m,r)$.

We modify and clarify the arguments appearing in the proof of Lemma 3.1 of \cite{dds}.  
We apply Lemma \ref{switch} for $R=K$, $M=End_{K}(V^{\otimes r})$, the basis $\mathcal{B}$ 
consisting of all elements $E_{IJ}$ for multi-indices $I=i_1\ldots i_r$ and $J=j_1\ldots j_r$ of length $r$, and 
$U=End_{K\Sigma_r}(V^{\otimes r})$.

Let $\phi\in U$ and $\phi=\sum_{IJ} a_{IJ} E_{IJ}$. 

If $1\leq j<r$, then 
\[\begin{aligned}\phi(v_L)\tau_j&=(\sum_{IJ} a_{IJ} E_{IJ}(v_L))\tau_j=(\sum_I a_{IL} v_I)\tau_j=\sum_I a_{IL} v_{I.(j,j+1)}\\&=\sum_I  a_{I.(j,j+1),L} v_{I}
\end{aligned}\]
and
\[\begin{aligned}
\phi(v_L \tau_j)&=\phi(v_{L.(j,j+1)})=\sum_{IJ} a_{IJ} E_{IJ}(v_{L.(j,j+1)})=\sum_I 
a_{I,L.(j,j+1)} v_I.
\end{aligned}
\]

Thus $\phi$ is invariant under the transposition $\tau_j$ if and only if $a_{I.(j,j+1),L}=a_{I,L.(j,j+1)}$ for each $I,L$. 
Therefore, $U$ is described by the set of equations in the coefficients $a_{IJ}$ of $\phi\in U$ as 
\[\phi\in U \text{ if and only if } a_{I.(j,j+1), J}=a_{I,J.(j,j+1)} \text{ for each } I, J  \text{ and } 1\leq j<r.\]

The basis $\mathcal{B}^*$ of $M^*$, dual to $\mathcal{B}$, consists of the coefficient functions $E_{IJ}^*=x_{IJ}$
that are constructed multiplicatively from the coefficient functions $x_{ij}=e_{ij}^*$, meaning that 
$x_{IJ}=x_{i_1j_1}\ldots x_{i_rj_r}$ for $I=i_1\ldots i_r$ and $J=j_1\ldots j_r$. 

The submodule $X$ of $M^*$ is generated by equations $x_{I.(j,j+1),J}=x_{I,J.(j.j+1)}$ for all $I,J$ and $1\leq j<r$.
Due to the multiplicativity of the coefficient functions, these equations are identical to the relations generated by the commutativity relations
$x_{ij}x_{kl}=x_{kl}x_{ij}$ for all $1\leq i,j,k,l\leq m$.

If $F$ is the free algebra on generators $x_{ij}$ for $1\leq i,j\leq m$ and $Y$ the submodule of $F$ generated 
by elements $x_{ij}x_{kl}-x_{kl}x_{ij}$ for all $1\leq i,j,k,l\leq m$, then $M^*/X$ is isomorphic to the degree $r$ component $D_r$ of the polynomial algebra $F/Y$ on generators $x_{ij}$. 

Then $D_r\simeq A(m,r)$ shows that $S(m,r)=A(m,r)^*$  has the same dimension as 
$End_{K\Sigma_r}(V^{\otimes r})^{op}$.
We conclude that $\Psi_r$ is surjective and $S(m,r)\simeq End_{K\Sigma_r}(V^{\otimes r})^{op}$.
\end{proof}

\section{The second half of the mixed Schur-Weyl duality for $GL(m)$}

Let $r,s\geq 0$ and $I,J,K$ be multi-indices of length $r+s$ with entries from the set $\{1, \ldots, m\}$.
We write every multi-index $I=(i_1, \ldots, i_{r+s})$ as a concatenation $I=I_VI_W$, where $I_V=(i_1 \ldots i_r)$ and $I_W=(i_{r+1} \ldots i_{r+s})$.

Let $v_1, \ldots, v_m$ be a $K$-basis of $V$ and $v^*_1, \ldots, v^*_m$ be the corresposponding dual basis of $W$.
Then the elements $v_{I}=v_{I_V}v_{I_W}=v_{i_1}\ldots v_{i_r}v^*_{i_{r+1}}\ldots v^*_{i_{r+s}}$ for all multi-indices $I$ of length $r+s$ form a basis of $V^{\otimes r}\otimes W^{\otimes s}$.

Since
\[\begin{aligned}&End_K(V^{\otimes r}\otimes W^{\otimes s})\simeq V^{\otimes r}\otimes W^{\otimes s}\otimes (V^{\otimes r}\otimes W^{\otimes s})^*
\simeq V^{\otimes r}\otimes W^{\otimes s}\otimes W^{\otimes r}\otimes V^{\otimes s}\\
&\simeq (V^{\otimes r}\otimes W^{\otimes r})\otimes (W^{\otimes s}\otimes V^{\otimes s})\simeq
End_K(V^{\otimes r})\otimes End_K(W^{\otimes s})\\
&\simeq End_K(V)^{\otimes r}\otimes End_K(W)^{\otimes s},\end{aligned}\]
every map from $End_R(V^{\otimes r}\otimes W^{\otimes s})$ is a linear combination of maps 
$\phi_{V^{\otimes r}}\otimes \phi_{W^{\otimes s}}$, where $\phi_{V^{\otimes r}}\in End_K(V^{\otimes r})$ and  $\phi_{W^{\otimes s}}\in End_K(W^{\otimes s})$. 

Denote by $E^V_{IJ}$ the matrix unit given by $E^V_{IJ}(v_{L_V})=\delta_{J_VL_V}v_{I_V}$, and 
by $E^W_{IJ}$ the matrix unit given by $E^W_{IJ}(v_{L_W})=\delta_{J_WL_W}v_{I_W}$.
Then the elements $E_{IJ}=E_{IJ}^VE_{IJ}^W$ form a $K$-basis
of the space $End_K(V^{\otimes r}\otimes W^{\otimes s})$. 
For $\phi\in End_K(V^{\otimes r}\otimes W^{\otimes s})$ we write 
$\phi=\sum_{IJ} a_{IJ} E_{IJ}$ for appropriate coefficients $a_{IJ}$.

\begin{lm}\label{reduction2}
Assume $\phi\in End_K(V^{\otimes r}\otimes W^{\otimes s})$ commutes with the action of $\Sigma_r$ on the first $r$ components and with the action of $\Sigma_s$ on the last $s$ components of $V^{\otimes r}\otimes W^{\otimes s}$.
If $p>\max\{r,s\}$, then $\phi$  can be written as a linear combination of maps
$(\psi_V\otimes \ldots \otimes \psi_V)\otimes (\psi_W\otimes \ldots \otimes \psi_W)$ for some $\psi_V\in End_K(V)$ and 
$\psi_W\in End_K(W)$. If $\psi_V=\sum_{ij} a_{ij} E_{ij}$, $\psi_W=\sum_{ij} a^*_{ij} E_{ij}$
and $(\phi_V\otimes\ldots\otimes \psi_V)\otimes (\psi_W\otimes \ldots \otimes \psi_W)=\sum_{IJ} a_{IJ} E_{IJ}$, then  $a_{IJ}=a_{i_1j_1}\ldots a_{i_rj_r}a^*_{i_{r+1}j_{r+1}}\ldots a^*_{i_{r+s}j_{r+s}}$. 
\end{lm}
\begin{proof}
Since
\[\begin{aligned}&End_K(V^{\otimes r}\otimes W^{\otimes s})\simeq
End_K(V^{\otimes r})\otimes End_K(W^{\otimes s})
\simeq End_K(V)^{\otimes r}\otimes End_K(W)^{\otimes s},\end{aligned}\]
every map from $End_R(V^{\otimes r}\otimes W^{\otimes s})$ is a linear combination of maps 
$\phi_{V^{\otimes r}}\otimes \phi_{W^{\otimes s}}$, where $\phi_{V^{\otimes r}}\in End_K(V^{\otimes r})$ and  $\phi_{W^{\otimes s}}\in End_K(W^{\otimes s})$. 
If $\phi\in End_K(V^{\otimes r}\otimes W^{\otimes s})$ commutes with the action of $\Sigma_r$ on the first $r$ components and with the action of $\Sigma_s$ on the last $s$ components, then its image in $End_K(V^{\otimes r})\otimes 
End_K(W^{\otimes s})$ belongs to $Sym^r(End_K(V))\otimes Sym^s(End_K(W))$.

If $p>r$ and $p>s$, then by Lemma \ref{reduction} the maps $\phi_{V^{\otimes r}}$ and $\phi_{W^{\otimes s}}$
are linear combinations of maps $\psi_V\otimes \ldots \otimes \psi_V$ and 
$\psi_W\otimes \ldots \otimes \psi_W$ for $\psi_V\in End_K(V)$ and 
$\psi_W\in End_K(W)$, respectively. The multiplicativity statement is obvious.
\end{proof}

To prove the next statement, we will not use the above lemma, but we modify arguments in Section 4 of \cite{dds} to the supercase. The partial multiplicativity will follow from the choice of the basis $\mathcal{B}$ in Lemma \ref{switch}.

\begin{pr}\label{mixedclassic}
There is  $S(m,r,s)\simeq End_{B_{r,s}(\delta)}(V^{\otimes r}\otimes W^{\otimes s})^{op}$.
\end{pr}
\begin{proof}
We apply Lemma \ref{switch} for $R=K$, $M=End_K(V^{\otimes r}\otimes W^{\otimes s})$, 
the basis $\mathcal{B}$ consisting of all elements $E_{IJ}$ for multi-indices $I=i_1\ldots i_r i_{r+1}\ldots i_{r+s}$, 
$J=j_1\ldots j_r j_{r+1}\ldots j_{r+s}$, and $U=End_{B_{r,s}(\delta)}(V^{\otimes r}\otimes W^{\otimes s})$.

Assume $\phi\in U$ and write $\phi=\sum_{IJ} a_{IJ} E_{IJ}$.
If $j\neq r$, then 
\[\begin{aligned}\phi(v_L)\tau_j&=(\sum_{IJ} a_{IJ} E_{IJ}(v_L))\tau_j=(\sum_I a_{IL} v_I)\tau_j=\sum_I  a_{IL} v_{I.(j,j+1)}\\&=\sum_I  a_{I.(j,j+1),L} v_{I}
\end{aligned}\]
and
\[\begin{aligned}
\phi(v_L \tau_j)&=\phi(v_{L.(j,j+1)})=\sum_{IJ} a_{IJ} E_{IJ}(v_{L.(j,j+1)})=\sum_I 
a_{I,L.(j,j+1)} v_I.
\end{aligned}
\]
Comparing coefficients at $v_I,$ we obtain that the above expressions coincide if and only if $a_{I.(j,j+1),L}=a_{I,L.(j,j+1)}$ for each $j\neq r$.

It remains to deal with $\tau_r$. For simplicity of writing, assume $r=s=1$ (similar arguments remain valid for all values of $r,s$.)
Then 
\[\phi(v_L)\tau_r=(\sum_{IJ} a_{IJ} E_{IJ}(v_L))\tau_r=(\sum_{I} a_{IL} v_I)\tau_r=\sum_{i} a_{(ii)L} \sum_{t=1}^{m} v_{(tt)}
 \]
and
\[\begin{aligned}\phi(v_L\tau_r)&=\delta_{l_1l_2}\sum_{t=1}^{m} \phi(v_{(tt)})=\delta_{l_1l_2}\sum_{t=1}^{m} \sum_I a_{I(tt)} E_{I(tt)}(v_{(tt)})\\
&=\delta_{l_1l_2}\sum_{t=1}^{m} \sum_{I} a_{I(tt)} v_I.
\end{aligned}
\]
The last two expressions coincide if and only if the following conditions are satisfied.

If $l_1\neq l_2$, then $\sum_i a_{il_1}\otimes a^*_{il_2}=0$ (comparing coefficients at $v_{(tt)}$).

If $l_1=l_2$ and $i_1\neq i_2$, then $\sum_{t=1}^{m} a_{i_1t}\otimes a^*_{i_2t}=0$ (comparing coefficients at $v_{(i_1i_2)}$).

If $l_1=l_2=l$ and $i_1=i_2$, then $\sum_i a_{il}\otimes a^*_{il}=\sum_{t=1}^{m} 
a_{i_1t}\otimes a^*_{i_1t}$ (comparing coefficients at $v_{(i_1i_1)}$).

The basis $\mathcal{B}^*$ of $M^*$, dual to $\mathcal{B}$, consists of coefficient functions $E_{IJ}^*=x_{IJ}$
that are constructed multiplicatively from the coefficient functions $x_{ij}=(e^V_{ij})^*$ and 
$x_{ij}^*=(e^W_{ij})^*$ for $1\leq i,j\leq m$ in the sense that 
\[x_{IJ}=x_{i_1j_1}\ldots x_{i_rj_r}x^*_{i_{r+1}j_{r+1}}\ldots x^*_{i_{r+s}j_{r+s}}\] for $I=i_1\ldots i_ri_{r+1}\ldots i_{r+s}$ and $J=j_1\ldots j_rj_{r+1}\ldots j_{r+s}$. 

Let us review the generating equations for the submodule $X$ of $M^*$ given by Lemma \ref{switch}.
Due to the multiplicativity of the coefficient functions $x_{IJ}$, the equations $x_{I.(j,j+1),J}=x_{I,J.(j.j+1)}$ for all $I,J$ and $1\leq j<r$ are identical to the relations generated by the commutativity relations 
$x_{ij}x_{kl}=x_{kl}x_{ij}$ for all $1\leq i,j,k,l\leq m$.
The same equations for $r\leq j<r+s$ are identical to the relations generated by the commutativity relations 
$x^*_{ij}x^*_{kl}=x^*_{kl}x^*_{ij}$ for all $1\leq i,j,k,l\leq m$.

Using the multiplicativity of $x_{IJ}$, we derive that the remaining equations for $X$ are generated by the equations 
$\sum_{k=1}^{m} x_{ki} x^*_{kj}=0$ for  $i\neq j$, 
$\sum_{k=1}^{m} x_{ik}x^*_{jk} = 0$  for  $i\neq j$
and $\sum_{k=1}^{m} x_{ki}x^*_{ki}=\sum_{k=1}^{m} x_{jk}x^*_{jk}$ for all  $i,j$.

Denote by $F=F_1\otimes_K F_2$, where $F_1$ the $K$-submodule of the free algebra on generators $x_{ij}$ generated by monomials of degree $r$, and $F_2$ the $K$-submodule of the free algebra on generators $x^*_{ij}$ generated by monomials of degree $s$, where $1\leq i,j\leq m$. 

Let $Y$ the submodule of $F$ generated by relations
\[x_{ij}x_{kl}= x_{kl}x_{ij} \text{ for all } i,j;\]
\[x^*_{ij}x^*_{kl}=x^*_{kl}x^*_{ij} \text{ for all } i,j;\]
\[\sum_{k=1}^{m} x_{ki} x^*_{kj}=0 \text{ for } i\neq j;\]
\[\sum_{k=1}^{m} x_{ik}x^*_{jk} = 0 \text{ for } i\neq j;\]
\[\sum_{k=1}^{m} x_{ki}x^*_{ki}=\sum_{k=1}^{m} x_{jk}x^*_{jk} \text { for all } i,j.\]

Then by Lemma \ref{switch}, $((End_{B_{r,s}(\delta)}(V^{\otimes r}\otimes W^{\otimes s}))^{op})^*$, as a $K$-module, is isomorphic to $F/Y$.

Recall the definition of $\tilde{A}(m;r,s)$ on p. 62 of \cite{dd}. We write $A(m,r,s)$ instead of $\tilde{A}(m;r,s)$ for simplicity.
Let $c_{ij}$ be the coefficient functions of the generic matrix $C$ and 
\[d_{ij}=(-1)^{l+k}\frac{C[1, \ldots, \widehat{k}, \ldots, m|1, \ldots, \widehat{l}, \ldots, m]}{C[1,\ldots, m|1, \ldots m]}\] be the coefficient functions of the matrix $C^{-1}$. Then $A(m,r,s)$ is the subspace (and a subcoalgebra) of $K[GL(m)]$ spanned by all products of the form $\prod_{ij} c_{ij}^{a_{ij}}\prod_{ij} d_{ij}^{b_{ij}}$
such that $a_{ij}, b_{ij}\geq 0$, $\sum_{ij} a_{ij}=r$ and $\sum_{ij} b_{ij}=s$.

The generators $c_{ij}$ and $d_{ij}$ commute and are subject to additional relations 
$\sum_{j=1}^mc_{ik}d_{kj}=\delta_{ij}$ and $\sum_{k=1}^m c_{ki}d_{jk}=\delta_{ij}$. 

In particular, 
\[\sum_{k=1}^mc_{ik}d_{kj}=0 \text{ and } \sum_{k=1}^m c_{ki}d_{jk}=0 \text{ if } i\neq j,\]
and 
\[\sum_{k=1}^mc_{ik}d_{ki}=\sum_{k=1}^m c_{kj}d_{jk}\text{ for all } 1\leq i,j\leq m.\]
Since we need to have $\sum_{ij} a_{ij}=r$ and $\sum_{ij} b_{ij}=s$, we cannot equate the terms  $\sum_{k=1}^mc_{ik}d_{ki}=\sum_{k=1}^m c_{kj}d_{jk}$ to 1 since it would reduce the corresponding degrees to less than $r$ and $s$.

Therefore, the $K$-space $F/Y$ is isomorphic to $A(m,r,s)$ via 
\[\prod_{ij}x_{ij}^{a_{ij}}\otimes \prod_{ij}(x^*_{ij})^{b_{ij}}\mapsto
\prod_{ij}c_{ji}^{a_{ij}}\prod_{ij}(d_{ij})^{b_{ij}}.
\]
Please note the use of the transposition $x_{ij}\mapsto c_{ji}$. Instead, we could have used a transposition in the second component $x^*_{ij}\mapsto d_{ji}$.

In particular, the dimension of $S(m,r,s)$ is the same of that of $A(m,r,s)$ and  
$End_{B_{r,s}(\delta)}(V^{\otimes r}\otimes W^{\otimes s}))^{op}$.

Therefore, the morphism $\Psi_{r,s}:Dist(G)\to End_{B_{r,s}(\delta)}(V^{\otimes r}\otimes W^{\otimes s}))^{op}$ 
is surjective, and $S(m,r,s)\simeq End_{B_{r,s}(\delta)}(V^{\otimes r}\otimes W^{\otimes s}))^{op}$.
\end{proof}

\section{The second half of the super Schur-Weyl duality}

The second part of super Schur-Weyl duality states that $S(m|n,r)$ is isomorphic to $End_{K\Sigma_r}(V^{\otimes r})^{op}$. 

Let $I,J$ be multi-indices of length $r$ with entries from the set $\{1, \ldots, m+n\}$.
Denote by $E_{IJ}$ the matrix unit given by $E_{IJ}(v_K)=\delta_{JK} v_I$. The elements $E_{IJ}$ form a $K$-basis
of the space $End_K(V^{\otimes r})$. 
For $\phi\in End_K(V^{\otimes r})$ we write 
$\phi=\sum_{IJ} a_{IJ} E_{IJ}$ for appropriate coefficients $a_{IJ}$.

\begin{pr}\label{classicsuper}
The Schur superalgebra $S(m|n,r)$ is isomorphic to 
$End_{K\Sigma_r}(V^{\otimes r})^{op}$. 
\end{pr}
\begin{proof}
We show that the dual of $End_{K\Sigma_r}(V^{\otimes r})^{op}$ is isomorphic to the superalgebra $A(m|n,r)$.

We modify the arguments from the proof Proposition \ref{classic}.
We apply Lemma \ref{switch} for $R=K$, $M=End_{K}(V^{\otimes r})$, the basis $\mathcal{B}$ 
consisting of all elements $E_{IJ}$ for multi-indices $I=i_1\ldots i_r$ and $J=j_1\ldots j_r$ of length $r$, and 
$U=End_{K\Sigma_r}(V^{\otimes r})$.

Let $\phi\in U$ and $\phi=\sum_{IJ} a_{IJ} E_{IJ}$. 

If $1\leq j<r$, then 
\[\begin{aligned}\phi(v_L)\tau_j&=(\sum_{IJ} a_{IJ} E_{IJ}(v_L))\tau_j=(\sum_I a_{IL} v_I)\tau_j\\
&=\sum_I (-1)^{|i_j||i_{j+1}|} a_{IL} v_{I.(j,j+1)}=\sum_I (-1)^{|i_j||i_{j+1}|} a_{I.(j,j+1),L} v_{I}
\end{aligned}\]
and
\[\begin{aligned}
\phi(v_L \tau_j)&=\phi((-1)^{|l_j||l_{j+1}|}v_{L.(j,j+1)})=(-1)^{|l_j||l_{j+1}|}\sum_{IJ} a_{IJ} E_{IJ}(v_{L.(j,j+1)})\\
&=\sum_I (-1)^{|l_j||l_{j+1}|}
a_{I,L.(j,j+1)} v_I.
\end{aligned}
\]

Thus $\phi$ is invariant under the transposition $\tau_j$ if and only if 
\[(-1)^{|i_j||i_{j+1}|}a_{I.(j,j+1),L}=(-1)^{|l_j||l_{j+1}|}a_{I,L.(j,j+1)}\] for each $I,L$. 
Therefore, $U$ is described by the set of equations in the coefficients $a_{IJ}$ of $\phi\in U$ as $\phi\in U$ if and only if  \[(-1)^{|i_j||i_{j+1}|}a_{I.(j,j+1), L}=(-1)^{|l_j||l_{j+1}|}a_{I,L.(j,j+1)} \text{ for each } I, L \text{ and } 1\leq j<r.\]

The basis $\mathcal{B}^*$ of $M^*$, dual to $\mathcal{B}$, consists of coefficient functions $E_{IJ}^*=x_{IJ}$
that are constructed multiplicatively from the coefficient functions $x_{ij}=e_{ij}^*$, meaning that 
$x_{IJ}=x_{i_1j_1}\ldots x_{i_rj_r}$ for $I=i_1\ldots i_r$ and $J=j_1\ldots j_r$. 

The submodule $X$ of $M^*$ is generated by equations 
\[(-1)^{|i_j||i_{j+1}|}x_{I.(j,j+1),L}=(-1)^{|l_j||l_{j+1}|}x_{I,L.(j,j+1)}\] for all $I,L$ and $1\leq j<r$.
Due to the multiplicativity of the coefficient functions, these equations are identical to the relations generated by the commutativity relations 
\[(-1)^{|i||k|}x_{ij}x_{kl}=(-1)^{|j||l|}x_{kl}x_{ij}\] for all $1\leq i,j,k,l\leq m$.

The last relations are not the usual supercommutativity relations between coordinate functions. To get 
the usual supercommutativity relations, we need to replace the basis $\mathcal{B}$ by a different basis. 
To motivate this switch, recall that the coefficient functions of $V^{\otimes r}$ for the basis consisting of $v_I$  are given as
\[\chi_{IJ}=(-1)^{\sum_{t=1}^r |i_t|(|i_{t+1}|+|j_{t+1}|+\ldots +|i_r|+|j_r|)} c_{IJ},\]
where $c_{IJ}=c_{i_1,j_1}\ldots c_{i_r,j_r}$ is the product of the matrix coefficient functions.
The functions $\chi_{IJ}$ are easier to work with because the comultiplication in terms of $\chi_{IJ}$ is nicer than in terms of $c_{IJ}$.

Define a different basis $\mathfrak{B}$ of $End_K(V^{\otimes r})$ consisting of elements 
\[F_{IJ}=(-1)^{\sum_{t=1}^r |i_t|(|i_{t+1}|+|j_{t+1}|+\ldots +|i_r|+|j_r|)} E_{IJ}\]
and write $\phi=\sum_{IJ} b_{IJ} F_{IJ}$.
Then $b_{IJ}=(-1)^{\sum_{t=1}^r |i_t|(|i_{t+1}|+|j_{t+1}|+\ldots +|i_r|+|j_r|)} a_{IJ}$.

If we set $F_{IJ}^*=y_{IJ}$ and $y_{ij}=f_{ij}^*$, the relation $x_{ij}x_{kl}=(-1)^{|i||k|+|j||l|}x_{kl}x_{ij}$ is equivalent to 
\[\begin{aligned}y_{ij}y_{kl}&=(-1)^{|i||k|+|i||l|}x_{ij}x_{kl}=(-1)^{|i||k|+|i||l|+|i||k|+|j||l|}x_{kl}x_{ij}
=(-1)^{(|i||l|+|j||l|)}x_{kl}x_{ij}\\
&=(-1)^{|i||l|+|j||l|+|k||i|+|k||j|}
y_{kl}y_{ij}=(-1)^{(|i|+|j|)(|k|+|l|)}y_{kl}y_{ij}\\
&=(-1)^{|y_{ij}||y_{kl}|}y_{kl}y_{ij}, 
\end{aligned}\]
which is the usual supercommutativity relation between $y_{ij}$ and $y_{kl}$.

Let $F$ be the free algebra on generators $y_{ij}$ for $1\leq i,j\leq m$ and $Y$ the submodule of $F$ generated 
by elements $y_{ij}y_{kl}-(-1)^{|y_{ij}||y_{kl}|}y_{kl}y_{ij}$ for all $1\leq i,j,k,l\leq m$.  Then $M^*/X$, as a $K$-vector space, is isomorphic to the degree $r$ component $D_r$ of the algebra $F/Y$.
Therefore, $(End_{K\Sigma_r}(V^{\otimes r}))^*$, considered as a $K$-space, is isomorphic to $D_r$
by Lemma \ref{switch}. 
On the other hand, $D_r\simeq A(m|n,r)$, as $K$-spaces, via the isomorphism
$\prod_{ij}  c_{ij}^{a_{ij}} \mapsto \prod_{ij}  y_{ij}^{a_{ij}}$. This shows that 
$S(m|n,r)=A(m|n,r)^*$  has the same dimension as $End_{K\Sigma_r}(V^{\otimes r})^{op}$.
We conclude that $\Psi_r$ is surjective and $S(m|n,r)\simeq End_{K\Sigma_r}(V^{\otimes r})^{op}$.
\end{proof}

\section{The second half of the mixed super Schur-Weyl duality}

Let $v_1, \ldots, v_{m+n}$ be a $K$-basis of $V$, and $v^*_1, \ldots, v^*_{m+n}$ be the corresponding dual basis of $W$. Recall that the standard matrix units $e_{ij}$ act on $V$ and $W$ as
$e_{ij}v_k = \delta_{jk}v_i$ and $e_{ij}v^*_k=-\delta_{ik}(-1)^{|i|(|i|+|j|)}v^*_j$.

Let $r,s\geq 0$ and $I,J,K$ be multi-indices of length $r+s$ with entries from the set $\{1, \ldots, m+n\}$.
We write every multi-index $I=(i_1, \ldots, i_{r+s})$ as a concatenation $I=I_VI_W$, where $I_V=(i_1 \ldots i_r)$ and $I_W=(i_{r+1} \ldots i_{r+s})$. The elements $v_{I}=v_{I_V}v_{I_W}=v_{i_1}\ldots v_{i_r}v^*_{i_{r+1}}\ldots v^*_{i_{r+s}}$ for all multi-indices $I$ of length $r+s$ form a basis of $V^{\otimes r}\otimes W^{\otimes s}$.

Since
\[\begin{aligned}&End_K(V^{\otimes r}\otimes W^{\otimes s})\simeq 
End_K(V^{\otimes r})\otimes End_K(W^{\otimes s})\simeq End_K(V)^{\otimes r}\otimes End_K(W)^{\otimes s},\end{aligned}\]
every map from $End_R(V^{\otimes r}\otimes W^{\otimes s})$ is a linear combination of maps 
$\phi_{V^{\otimes r}}\otimes \phi_{W^{\otimes s}}$, where $\phi_{V^{\otimes r}}\in End_K(V^{\otimes r})$ and  $\phi_{W^{\otimes s}}\in End_K(W^{\otimes s})$. 

Denote by $E^V_{IJ}$ the matrix unit given by $E^V_{IJ}(v_{L_V})=\delta_{J_VL_V}v_{I_V}$, and 
by $E^W_{IJ}$ the matrix unit given by $E^W_{IJ}(v_{L_W})=\delta_{J_WL_W}v_{I_W}$.
Then the elements $E_{IJ}=E_{IJ}^VE_{IJ}^W$ form a $K$-basis
of the space $End_K(V^{\otimes r}\otimes W^{\otimes s})$. 
For $\phi\in End_K(V^{\otimes r}\otimes W^{\otimes s})$ we write 
$\phi=\sum_{IJ} a_{IJ} E_{IJ}$ for appropriate coefficients $a_{IJ}$.

We want to determine the dual of $End_{B_{r,s}(\delta)}(V^{\otimes r} \otimes W^{\otimes s})^{op}$
and compare it with $A(m|n,r,s)$.

\subsection{$(End_{B_{r,s}(\delta)}(V^{\otimes r} \otimes W^{\otimes s})^{op})^*$ expressed by generators and relations}
We apply Lemma \ref{switch} for $R=K$, $M=End_K(V^{\otimes r}\otimes W^{\otimes s})$, 
the basis $\mathcal{B}$ consisting of all elements $E_{IJ}$ for multi-indices $I=i_1\ldots i_r i_{r+1}\ldots i_{r+s}$, 
$J=j_1\ldots j_r j_{r+1}\ldots j_{r+s}$, and $U=End_{B_{r,s}(\delta)}(V^{\otimes r}\otimes W^{\otimes s})$.

Assume $\phi\in U$ and write $\phi=\sum_{IJ} a_{IJ} E_{IJ}$.
If $j\neq r$, then 
\[\begin{aligned}\phi(v_L)\tau_j&=(\sum_{IJ} a_{IJ} E_{IJ}(v_L))\tau_j=(\sum_I a_{IL} v_I)\tau_j\\
&=\sum_I (-1)^{|i_j||i_{j+1}|} a_{IL} v_{I.(j,j+1)}=\sum_I (-1)^{|i_j||i_{j+1}|} a_{I.(j,j+1),L} v_{I}
\end{aligned}\]
and
\[\begin{aligned}
\phi(v_L \tau_j)&=\phi((-1)^{|l_j||l_{j+1}|}v_{L.(j,j+1)})=(-1)^{|l_j||l_{j+1}|}\sum_{IJ} a_{IJ} E_{IJ}(v_{L.(j,j+1)})\\
&=\sum_I (-1)^{|l_j||l_{j+1}|}
a_{I,L.(j,j+1)} v_I.
\end{aligned}
\]
Comparing coefficients at $v_I$, we obtain that the above expressions coincide if and only if 
\[ (-1)^{|i_j||i_{j+1}|} a_{I.(j,j+1),L}= (-1)^{|l_j||l_{j+1}|}a_{I,L.(j,j+1)}\]
for each $j\neq r$.

It remains to deal with $\tau_r$. For simplicity of writing, assume $r=s=1$ (analogous arguments remain valid for all values of $r,s$.)
Then 
\[\phi(v_L)\tau_r=(\sum_{IJ} a_{IJ} E_{IJ}(v_L))\tau_r=(\sum_{I}a_{IL} v_I)\tau_r=\sum_{i} (-1)^{|i|+1}a_{(ii)L} \sum_{t=1}^{m+n} v_{(tt)}
 \]
and
\[\begin{aligned}\phi(v_L\tau_r)&=\delta_{l_1l_2}(-1)^{|l_1|+1}\sum_{t=1}^{m+n}\phi(v_{(tt)})=\delta_{l_1l_2}
(-1)^{|l_1|+1}\sum_{t=1}^{m+n}\sum_I a_{I(tt)} E_{I(tt)}(v_{(tt)})\\
&=\delta_{l_1l_2}(-1)^{|l_1|+1}\sum_{t=1}^{m+n} \sum_{I} a_{I(tt)} v_I.
\end{aligned}
\]
The last two expressions coincide if and only if the following conditions are satisfied.

If $l_1\neq l_2$, then $\sum_i (-1)^{|i|}a_{il_1}\otimes a^*_{il_2}=0$ (comparing coefficients at $v_{(tt)}$).

If $l_1=l_2=l$ and $i_1\neq i_2$, then $(-1)^{|l|}\sum_{t=1}^{m+n} a_{i_1t}\otimes a^*_{i_2t}=0$ (comparing coefficients at $v_{(i_1i_2)}$).

If $l_1=l_2=l$ and $i_1=i_2$, then $\sum_i (-1)^{|i|} a_{il}\otimes a^*_{il}= (-1)^{|l|}\sum_{t=1}^{m+n}
a_{i_1t}\otimes a^*_{i_1t}$ (comparing coefficients at $v_{(i_1i_1)}$).

The basis $\mathcal{B}^*$ of $M^*$, dual to $\mathcal{B}$, consists of coefficient functions $E_{IJ}^*=x_{IJ}$
that are constructed multiplicatively from the coefficient functions $x_{ij}=(e^V_{ij})^*$ and 
$x_{ij}^*=(e^W_{ij})^*$ for $1\leq i,j\leq m+n$ in the sense that 
\[x_{IJ}=x_{i_1j_1}\ldots x_{i_rj_r}x^*_{i_{r+1}j_{r+1}}\ldots x^*_{i_{r+s}j_{r+s}}\] for $I=i_1\ldots i_ri_{r+1}\ldots i_{r+s}$ and $J=j_1\ldots j_rj_{r+1}\ldots j_{r+s}$. 

Next, we look at the generating equations for the submodule $X$ of $M^*$ given by Lemma \ref{switch}.
Due to the multiplicativity of the coefficient functions $x_{IL}$, the equations 
\[(-1)^{|i_j||i_{j+1}|}x_{I.(j,j+1),L}=(-1)^{|l_j||l_{j+1}|}x_{I,J.(j.j+1)}\] for all $I,L$ and $1\leq j<r$ are identical to the relations generated by the commutativity relations 
\[(-1)^{|i_j||i_{j+1}|} x_{i_{j+1}l_j}x_{i_jl_{j+1}}=(-1)^{|l_j||l_{j+1}|} x_{i_jl_{j+1}}x_{i_{j+1}l_j}
\]
which are equivalent to  
\[x_{ij}x_{kl}=(-1)^{|i||k|+|j||l|} x_{kl}x_{ij}\] for all $1\leq i,j,k,l\leq m+n$.
The same equations for $r\leq j<r+s$ are identical to the relations generated by the commutativity relations 
\[x^*_{ij}x^*_{kl}=(-1)^{|i||k|+|j||l|} x^*_{kl}x^*_{ij}\] for all $1\leq i,j,k,l\leq m+n$.

Using the multiplicativity of $x_{IL}$, we derive that the remaining equations for $X$ are generated by the equations 
$\sum_{k=1}^{m+n} (-1)^{|k|}x_{ki} x^*_{kj}=0$ for  $i\neq j$, 
$\sum_{k=1}^{m+n} x_{ik}x^*_{jk} = 0$  for  $i\neq j$,
and $\sum_{k=1}^{m+n} (-1)^{|k|}x_{ki}x^*_{ki}=(-1)^{|i|}\sum_{k=1}^{m+n} x_{jk}x^*_{jk}$ for all  $1\leq i,j \leq m+n$.

Using Lemma \ref{switch}, we establish the following statement.

\begin{pr}\label{mixedsuper}
Denote by $F=F_1\otimes_K F_2$, where $F_1$ is the $K$-submodule of the free algebra on generators $x_{ij}$ generated by monomials of degree $r$, and $F_2$ is the $K$-submodule of the free algebra on generators $x^*_{ij}$ generated by monomials of degree $s$, where $1\leq i,j\leq m+n$. 

Let $Y$ the submodule of $F$ generated by relations
\[x_{ij}x_{kl}=(-1)^{|i||k|+|j||l|} x_{kl}x_{ij} \text{ for all } i,j;\]
\[x^*_{ij}x^*_{kl}=(-1)^{|i||k|+|j||l|} x^*_{kl}x^*_{ij} \text{ for all } i,j;\]
\[\sum_{k=1}^{m+n} (-1)^{|k|}x_{ki} x^*_{kj}=0 \text{ for } i\neq j;\]
\[\sum_{k=1}^{m+n}x_{ik}x^*_{jk} = 0 \text{ for } i\neq j;\]
\[\sum_{k=1}^{m+n} (-1)^{|k|}x_{ki}x^*_{ki}=(-1)^{|i|}\sum_{k=1}^{m+n}  x_{jk}x^*_{jk} \text { for all } i,j.\]

Then $((End_{B_{r,s}(\delta)}(V^{\otimes r}\otimes W^{\otimes s}))^{op})^*$,
as a $K$-space, is isomorphic to $F/Y$.
\end{pr}

\subsection{Mixed super Schur-Weyl duality}

\begin{pr} There is 
$S(m|n,r,s)\simeq End_{B_{r,s}(\delta)}(V^{\otimes r}\otimes W^{\otimes s})^{op}$.
\end{pr}
\begin{proof}
Analogously to the second half of the proof of Proposition \ref{classicsuper}, we define a different basis of $End_K(V^{\otimes r})$ consisting of elements 
\[F^V_{IJ}=(-1)^{\sum_{t=1}^{r} |i_t|(|i_{t+1}|+|j_{t+1}|+\ldots +|i_{r}|+|j_{r}|)}
 E^V_{IJ},\]
a different basis of $End_K(W^{\otimes s})$ consisting of elements 
\[F^W_{IJ}=(-1)^{\sum_{t=r+1}^{r+s} |i_t|(|i_{t+1}|+|j_{t+1}|+\ldots +|i_{r+s}|+|j_{r+s}|)}
 E^W_{IJ},\]
which gives a different basis of $End_K(V^{\otimes r}\otimes W^{\otimes s})$ consisting of elements
$F_{IJ}=F^V_{IJ}F^W_{IJ}$.

If we write $\phi\in End_K(V^{\otimes r}\otimes W^{\otimes s})$ as $\phi=\sum_{IJ} b_{IJ} F_{IJ}$, 
then  
\[b_{IJ}=(-1)^{\sum_{t=1}^{r} |i_t|(|i_{t+1}|+|j_{t+1}|+\ldots +|i_{r}|+|j_{r}|)+\sum_{t=r+1}^{r+s} |i_t|(|i_{t+1}|+|j_{t+1}|+\ldots +|i_{r+s}|+|j_{r+s}|)}a_{IJ}.\]

We define $y_{ij}=(f^V_{ij})^*$, $y^*_{ij}=(f^W_{ij})^*$ for $1\leq i,j\leq m+n$ and 
\[y_{IJ}=(F_{IJ})^*=y_{i_1j_1}\ldots y_{i_rj_r}y^*_{i_{r+1}j_{r+1}}\ldots y^*_{i_{r+s}j_{r+s}}.\]

In the same way as in the second half of the proof of Proposition \ref{classicsuper}, we obtain that the relation
$x_{ij}x_{kl}=(-1)^{|i||k|+|j||l|} x_{kl}x_{ij}$
is equivalent to
\[\begin{aligned}y_{ij}y_{kl}&=(-1)^{|y_{ij}||y_{kl}|}y_{kl}y_{ij}, 
\end{aligned}\]
which is the usual supercommutativity relations between $y_{ij}$ and $y_{kl}$.
Analogously, the relation 
$x^*_{ij}x^*_{kl}=(-1)^{|i||k|+|j||l|} x^*_{kl}x^*_{ij}$
is equivalent to
\[\begin{aligned}y^*_{ij}y^*_{kl}&=(-1)^{|y^*_{ij}||y^*_{kl}|}y^*_{kl}y^*_{ij}.
\end{aligned}\]

The remaining relations are rewritten in terms of $y_{ij}$'s and $y^*_{ij}$'s as follows.

\[\begin{aligned}&\sum_{k=1}^{m+n} (-1)^{|i|(|k|+1)}y_{ki}y^*_{kj}=\sum_{k=1}^{m+n}(-1)^{|k|+|i|}(-1)^{|k|(|k|+|i|)}y_{ki}y^*_{kj}\\
&=(-1)^{|i|}\sum_{k=1}^{m+n} (-1)^{|k|} x_{ki}x^*_{kj}=0
\end{aligned}\]
for $i\neq j$,

\[\sum_{k=1}^{m+n} (-1)^{|j|(|i|+|k|)}y_{ij}y^*_{jk}=\sum_{k=1}^{m+n} x_{ik}x^*_{jk}=0\]
for $i\neq j$, and 

\[\begin{aligned}&\sum_{k=1}^{m+n} (-1)^{|i|(|k|+1)}y_{ki}y^*_{ki}=\sum_{k=1}^{m+n}(-1)^{|k|+|i|}(-1)^{|k|(|k|+|i|)}y_{ki}y^*_{ki}\\
&=(-1)^{|i|}\sum_{k=1}^{m+n} (-1)^k x_{ki}x^*_{ki}=\sum_{k=1}^{m+n} x_{jk}x^*_{jk}
=\sum_{k=1}^{m+n} (-1)^{|j|(|j|+|k|)} y_{jk}y^*_{jk}
\end{aligned}\]
for every $i,j$.

Therefore, $((End_{B_{r,s}(\delta)}(V^{\otimes r}\otimes W^{\otimes s}))^{op})^*$,
as a $K$-space, is isomorphic to $F/Y$, where $F$ is as above, and 
$Y$ is generated by the supercommutativity relations
\[y_{ij}y{kl}=(-1)^{|y_{ij}||y_{kl}|} y_{kl}y_{ij} \text{ for all } i,j;\]
\[y^*_{ij}y^*_{kl}=(-1)^{|y^*_{ij}||y^*_{kl}|} y^*_{kl}y^*_{ij} \text{ for all } i,j\]
and the additional relations 
\[\sum_{k=1}^{m+n} (-1)^{|i|(|k|+1)}y_{ki}y^*_{kj}=0 \text{ for } i\neq j;\]
\[\sum_{k=1}^{m+n} (-1)^{|j|(|i|+|k|)}y_{ij}y^*_{jk}= 0 \text{ for } i\neq j;\]
\[\sum_{k=1}^{m+n} (-1)^{|i|(|k|+1)}y_{ki}y^*_{ki}=\sum_{k=1}^{m+n} (-1)^{|j|(|j|+|k|)} y_{jk}y^*_{jk} \text { for all } i,j.\]

We will show that the $K$-spaces  $F/Y$ and  $A(m|n,r,s)$ are isomorphic.

Recall that we have denoted by $c_{ij}$ the coefficient functions of the generic matrix $C$ and by
\[d_{ij}=(-1)^{l+k}\frac{C[1, \ldots, \widehat{k}, \ldots, m|1, \ldots, \widehat{l}, \ldots, m]}{C[1,\ldots, m|1, \ldots m]}\] the coefficient functions of the matrix $C^{-1}$. Then $A(m|n,r,s)$ is the subspace (and a subcoalgebra) of $K[GL(m|n)]$ spanned by all products of the form $\prod_{ij} c_{ij}^{a_{ij}}\prod_{ij} d_{ij}^{b_{ij}}$
such that $a_{ij}, b_{ij}\geq 0$, $\sum_{ij} a_{ij}=r$ and $\sum_{ij} b_{ij}=s$.

The generators $c_{ij}$ and $d_{ij}$ supercommute and are subject to additional relations 
$\sum_{j=1}^mc_{ik}d_{kj}=\delta_{ij}$ and $\sum_{k=1}^m d_{jk}c_{ki}=\delta_{ij}$.

In particular, 
\[\sum_{k=1}^mc_{ik}d_{kj}=0 \text{ and  } \sum_{k=1}^m (-1)^{|c_{ki}||d_{jk}|}c_{ki}d_{jk}=0 \text{ if } i\neq j,\]
and 
\[\sum_{k=1}^mc_{ik}d_{ki}=\sum_{k=1}^m (-1)^{|c_{kj}||d_{jk}|} c_{kj}d_{jk} \text{ for all } 1\leq i,j\leq m.\] Since we need to have $\sum_{ij} a_{ij}=r$ and $\sum_{ij} b_{ij}=s$, we cannot equate the terms  $\sum_{k=1}^mc_{ik}d_{ki}=\sum_{k=1}^m c_{kj}d_{jk}$ to 1 since it would reduce the corresponding degrees to less than $r$ and $s$.

The isomorphism of $K$-spaces $F/Y$ and $A(m|n,r,s)$ is given using a supertransposition
$y_{ij}\mapsto  (-1)^{|i|(|j|+1)}c_{ji}$ and $y^*_{ij}\mapsto d_{ij}$.

More explicitly, the above-defined map extends multiplicatively to an isomorphism   
\[\prod_{ij}y_{ij}^{a_{ij}}\otimes \prod_{ij}(y^*_{ij})^{b_{ij}}\mapsto
\prod_{ij}( (-1)^{|i|(|j|+1)}c_{ji})^{a_{ij}}\prod_{ij}(d_{ij})^{b_{ij}}.
\]

Thus, the dimension of $S(m|n,r,s)$ is the same of that of $A(m|n,r,s)$ and  
$End_{B_{r,s}(\delta)}(V^{\otimes r}\otimes W^{\otimes s}))^{op}$.

Therefore, the morphism $\Psi_{r,s}:Dist(GL(m|n))\to End_{B_{r,s}(\delta)}(V^{\otimes r}\otimes W^{\otimes s}))^{op}$ 
is surjective, and $S(m|n,r,s)\simeq End_{B_{r,s}(\delta)}(V^{\otimes r}\otimes W^{\otimes s}))^{op}$.
\end{proof}

\section{The first half of the Schur-Weyl duality}

\subsection{Morphisms $\Phi_{r}$ and $\Phi_{r,s}$}

Recall the definitions of $\Phi_{r}$ and $\Phi_{r,s}$ from Section \ref{section1}.

\begin{tr}\label{tr1}
The morphism $\Phi_{r}$ is surjective. It is injective if and only if $r\leq m$.
\end{tr}
\begin{proof}
The statement is the first part of the classical Schur-Weyl duality. See Propositions 11 and 15 of \cite{d} for proof.
\end{proof}
The description of the kernel of $\Phi_r$ is given in Lemma 3 of \cite{h}.

\begin{theorem} (Theorem 7.8 of \cite{bs}) Assume the characteristic of the ground field $K$ is zero. Then the map $\Phi_{r,s}$ is surjective. It is injective if and only if $r+s<(m+1)(n+1)$.
\end{theorem}

To prove the above theorems using elementary methods analogous to those used earlier, 
we first need to describe $End_{G}(V^{\otimes r})$ and $End_{G}(V^{\otimes r}\otimes 
W^{\otimes s})$.
However, even finding the dimensions of these spaces is a nontrivial problem. The action of $K\Sigma_r$ on $V^{\otimes r}$ is faithful if and only if $r\geq m$.
Therefore, the dimension of $End_G(V^{\otimes r})$ is $r!$ if $r\geq m$, but it is unclear what this dimension is if $r<m$.
For example, if $m=2$ and $r=3$, this dimension is $5$. If $m\geq r=3$, the dimension is $r!=6$.
However, even for $m=r=3$, a direct verification involves a fair amount of computations. Similar obstacles appear in the description of the dimensions of 
$End_G(V^{\otimes r}\otimes W^{\otimes s})$.

For simplicity, we consider only $End_{GL(m)}(V^{\otimes r})$ and assume that the characteristic of $K$ is zero.

An element of $\phi\in End_{GL(m)}(V^{\otimes r})$ is an element $\phi$ of $End_{K}(V^{\otimes r})$ invariant under the action of standard matrix units $e_{ij}$ from the general linear Lie algebra $\mathfrak{gl}(m)$ for all $1\leq i,j \leq n$.

We write $\phi=\sum_{IJ}a_{IJ} E_{IJ}$ and evaluate 

\[\begin{aligned}&e_{ij}\phi(v_L)\\&=e_{ij}\sum_{IJ}a_{IJ} E_{IJ}(v_L)=e_{ij}\sum_{I}a_{IL}v_I=\sum_I a_{IL} e_{ij}v_I\\
&=
\sum_I a_{IL} \sum_{t=1}^r v_{i_1}\otimes \ldots \otimes e_{ij}v_{i_t}\otimes \ldots  \otimes v_{i_r}\\
&=\sum_I a_{IL} \sum_{t=1}^r \delta_{i_t,j} v_{i_1}\otimes \ldots \otimes v_{i}\otimes \ldots  \otimes v_{i_r}
=\sum_I a_{IL}\sum_{t=1}^r \delta_{i_t,j}  v_{i_1 \ldots i_{t-1}ii_{t+1}\ldots i_r}
\end{aligned}\]
and 
\[\begin{aligned}\phi(e_{ij}v_L) &= \phi(\sum_{t=1}^r v_{l_1}\otimes \ldots \otimes e_{ij}v_{l_t}\otimes \ldots \otimes v_{l_r})\\
&= \phi(\sum_{t=1}^r \delta_{l_t,j} v_{l_1}\otimes \ldots \otimes v_i\otimes \ldots \otimes v_{l_r})\\
&= \sum_{t=1}^r \delta_{l_t,j} \sum_{IJ} a_{IJ} E_{IJ}(v_{l_1}\otimes \ldots \otimes v_i\otimes \ldots \otimes v_{l_r})\\
&= \sum_{t=1}^r \delta_{l_t,j} \sum_{I} a_{I(l_1\ldots l_{t-1}il_{t+1}\ldots l_r)} v_I.
\end{aligned}\]

The equality $e_{ij}\phi(v_L)=\phi(e_{ij}v_L)$ is satisfied if and only if, for each multi-index $I$, the coefficients at $v_I$'s in the above two expressions are the same. 

Comparing coefficients is a daunting task, and the computations are generally intractable. We explain what needs to be done in the example below.

\begin{ex}
Assume $m=2$, $r=2$. 
Then $End_{GL(2)}(V^{\otimes 2})\simeq K\Sigma_2$.
\end{ex}
\begin{proof}
Consider cases

\underline{$i=j=1$}

$l_1=l_2=1$:

$E_1=2a_{11|11}v_{11}+a_{12|11}v_{12}+a_{21|11}v_{21}$

$E_2=2a_{11|11}v_{11}+2a_{12|11}v_{12}+2a_{21|11}v_{21}+2a_{22|11}v_{22}$

implies $a_{12|11}=a_{21|11}=a_{22|11}=0$.

$l_1=1$, $l_2=2$:

$E_1=2a_{11|12}v_{11}+a_{12|12}v_{12}+a_{21|12}v_{21}$

$E_2=a_{11|12}v_{11}+a_{12|12}v_{12}+a_{21|12}v_{21}+a_{22|12}v_{22}$

implies $a_{11|12}=a_{22|12}=0$.

$l_1=2$, $l_2=1$:

$E_1=2a_{11|21}v_{11}+a_{12|21}v_{12}+a_{21|21}v_{21}$

$E_2=a_{11|21}v_{11}+a_{12|21}v_{12}+a_{21|21}v_{21}+a_{22|21}v_{22}$

implies $a_{11|21}=a_{22|21}=0$.

$l_1=l_2=2$:

$E_1=2a_{11|22}v_{11}+a_{12|22}v_{12}+a_{21|22}v_{21}$

$E_2=0$

implies $a_{11|22}=a_{12|22}=a_{21|22}=0$.

That means only values unassigned are $a_{11|11}$, $a_{12|12}$, $a_{21|12}$, $a_{12|21}$, $a_{21|21}$ and $a_{22|22}$.

\underline{$i=1$,$j=2$}

$l_1=l_2=1$:

$E_1=a_{12|11}v_{11}+a_{21|11}v_{11}+a_{22|11}v_{12}+a_{22|11}v_{21}=0$

$E_2=0$

$l_1=1$, $l_2=2$:

$E_1=a_{12|12}v_{11}+a_{21|12}v_{11}+a_{22|12}v_{12}+a_{22|12}v_{21}$

$E_2=a_{11|11}v_{11}+a_{12|11}v_{12}+a_{21|11}v_{21}+a_{22|11}v_{22}=a_{11|11}v_{11}$

implies $a_{12|12}+a_{21|12}=a_{11|11}$

$l_1=2$, $l_2=1$:

$E_1=a_{12|21}v_{11}+a_{21|21}v_{11}+a_{22|21}v_{12}+a_{22|21}v_{21}$

$E_2=a_{11|11}v_{11}+a_{12|11}v_{12}+a_{21|11}v_{21}+a_{22|11}v_{22}=a_{11|11}v_{11}$

implies $a_{12|21}+a_{21|21}=a_{11|11}$

$l_1=l_2=2$:

$E_1=a_{12|22}v_{11}+a_{21|22}v_{11}+a_{22|22}v_{12}+a_{22|22}v_{21}=a_{22|22}v_{21}$

$E_2=a_{11|12}v_{11}+a_{11|21}v_{11}+a_{12|12}v_{12}+a_{12|21}v_{12}+a_{21|12}v_{21}+a_{21|21}v_{21}
+a_{22|12}v_{22}+a_{22|21}v_{22}$

implies $a_{12|12}+a_{12|21}=a_{22|22}$ and $a_{21|12}+a_{21|21}=a_{22|22}$.

As a consequence, we also get $a_{11|11}=a_{22|22}$, $a_{21|12}=a_{12|21}$ and $a_{12|12}=a_{21|21}$.

The general solution of  
$a_{12|12}+a_{12|21}=a_{11|11}$ is a linear combination of two solutions:

$a_{12|12}=a_{12|21}=1$ and $a_{11|11}=2$ (this corresponds to symmetric tensors) and

$a_{12|12}=1$, $a_{12|21}=-1$ and $a_{11|11}=0$  (this corresponds to antisymmetric tensors). 

Thus $End_{GL(2)}(V^{\otimes 2})\simeq K\Sigma_2$ as expected.
\end{proof}

For comparison, over arbitrary $GL(m)$, there is a well-known decomposition of $V^{\otimes 2}\simeq Sym^2(V)\oplus \Lambda^2(V)$ into irreducible symmetric and antisymmetric tensors.

\section*{Acknowledgement}
The author thanks Friederike Stoll for the conversation regarding the paper ``Quantized mixed tensor space and Schur-Weyl duality.''

\end{document}